\documentclass[11pt,english,reqno]{amsart}

\usepackage{amsthm,amsmath,amssymb,amsfonts}
\usepackage{booktabs,float,caption,setspace}
\usepackage{mathtools,mleftright}
\usepackage[usenames]{color}
\usepackage{a4wide}

\definecolor{webgreen}{rgb}{0,.5,0}
\definecolor{webbrown}{rgb}{.6,0,0}

\usepackage[
  pdfauthor={},
  pdfkeywords={},
  pdftitle={},
  pdfcreator={},
  pdfproducer={},
  linktocpage,colorlinks,bookmarksnumbered,linkcolor=blue,
  citecolor=webgreen,urlcolor=webbrown]{hyperref}

\theoremstyle{plain}
\newtheorem{theorem}{Theorem}

\newtheorem{corl}[theorem]{Corollary}
\newtheorem{prop}[theorem]{Proposition}
\newtheorem{lemma}[theorem]{Lemma}

\theoremstyle{definition}

\numberwithin{equation}{section}
\numberwithin{theorem}{section}
\numberwithin{figure}{section}
\numberwithin{table}{section}

\newcommand{\NN}{\mathbb{N}}
\newcommand{\ZZ}{\mathbb{Z}}
\newcommand{\QQ}{\mathbb{Q}}

\newcommand{\set}[1]{\left\{#1\right\}}

\newcommand{\valueat}[1]{{\,}_{\big|\, #1}}
\newcommand{\opdf}[2]{\nabla_{\hspace{-0.5ex}#1}^{\hspace{-0.03ex}#2}}
\newcommand{\opsh}{\eth}
\newcommand{\andq}{\quad \text{and} \quad}
\newcommand{\textq}[1]{\quad \text{#1} \quad}

\newcommand{\bb}{\mathbf{b}}
\newcommand{\BN}{\mathbf{B}}
\newcommand{\BNP}{\widehat{\BN}}
\newcommand{\BNPD}{\widehat{\overline{\BN}}}

\newcommand{\BNN}{\overline{\mathcal{B}}}
\newcommand{\PT}{\widetilde{\psi}}
\newcommand{\QT}{\widetilde{Q}}
\newcommand{\SH}{\widehat{S}}
\newcommand{\WQ}{\mathcal{W}}

\newcommand{\arxiv}[1]{\href{https://arxiv.org/abs/#1}{arXiv:#1}}

\makeatletter
\newcommand{\spod}[1]{\allowbreak\if@display\mkern8mu\else\mkern4mu\fi(#1)}
\newcommand{\smod}[1]{\spod{{\operator@font mod}\mkern6mu#1}}
\makeatother

\begin{document}

\title[Wilson's theorem modulo higher prime powers III]
{Wilson's theorem modulo higher prime powers III:\\
The cases modulo $p^6$ and $p^7$} 
\author{Bernd C. Kellner}
\address{G\"ottingen, Germany}
\email{bk@bernoulli.org}

\begin{abstract}
Extending previous work of the author, we compute the Wilson quotient modulo 
$p^5$ and $p^6$, and equivalently $(p-1)!$ modulo $p^6$ and $p^7$. Further, 
we determine some power sums of the Fermat quotients up to modulo $p^6$.
Subsequently, we discuss some patterns that occur in the $p$-adic coefficients
of the Wilson quotient as well as of $(p-1)!$, whereby the original congruence
$(p-1)! \equiv -1 \pmod{p}$ fits perfectly into the theory. 
\end{abstract}

\maketitle


\section{Introduction}

Let $p$ be always a prime. Wilson's theorem states that
\[
  (p-1)! \equiv -1 \smod{p},
\]
and the Wilson quotient is defined by
\[
  \WQ_p = \frac{(p-1)! + 1}{p}.
\]
We extend our previous work to compute supercongruences of $\WQ_p$ and $(p-1)!$
modulo higher prime powers of $p$. For the detailed theory, we refer to 
\cite{Kellner:2025} and \cite{Kellner:2025b}. 

For $n \geq 1$ and any prime $p \geq 7$, define the divided Bernoulli numbers 
\begin{equation} \label{eq:bnn-1}
  \BNN_n = \frac{\BN_{n(p-1)}+\frac{1}{p}-1}{n(p-1)}, 
\end{equation}
which are $p$-integral and satisfy the Kummer congruences 
\[
  \BNN_n \equiv \BNN_m \pmod{p} \quad (n, m \geq 1).
\]
Similarly, define
\begin{equation} \label{eq:bnn-2}
  \BNN_{n,d} = \frac{\BN_{n(p-1)-d}}{n(p-1)-d} \textq{with} d \in \set{2,4}\!,
\end{equation}
which are also $p$-integral and obey the Kummer congruences. The Bernoulli numbers 
are introduced in the next section. An \textsl{overbar}, e.g., $\BNN_n$, always 
indicates a \textsl{divided} Bernoulli number. The main results of the paper 
extend the results of \cite{Kellner:2025b}, as follows.

\begin{theorem} \label{thm:main}
Let $p \geq 7$ be a prime. Then we have
\begin{align*}
  \WQ_p &\equiv \sum_{\nu=1}^{5} \omega_\nu \, p^{\nu-1} \smod{p^5} \\
\shortintertext{and equivalently}
  (p-1)! &\equiv \sum_{\nu=0}^{5} \omega_\nu \, p^\nu \smod{p^6},
\end{align*}
where
\begin{align*}
  \omega_0 &= -1, \\
  \omega_1 &\equiv -5 \BNN_1 + 10 \BNN_2 - 10 \BNN_3 + 5 \BNN_4 - \BNN_5 \smod{p^5}, \\
  \omega_2 &\equiv -\tfrac{5}{2} \BNN_1^2 + \tfrac{15}{2} \BNN_2^2
  + \tfrac{5}{2} \BNN_3^2 + \BNN_1 \BNN_4 - 9 \BNN_2 \BNN_3 \smod{p^4}, \\
  \omega_3 &\equiv -\tfrac{1}{2} \BNN_1 \BNN_2^2
  - \BNN_1^2 (\tfrac{5}{3} \BNN_1 - \tfrac{5}{2} \BNN_2 + \tfrac{1}{2} \BNN_3)
  - \BNN_{1,2} + \BNN_{2,2} - \tfrac{1}{3} \BNN_{3,2} \smod{p^3}, \\
  \omega_4 &\equiv -\tfrac{5}{24} \BNN_1^4 + \tfrac{1}{6} \BNN_1^3 \BNN_2
  -\tfrac{2}{3} \BNN_1 \BNN_{1,2} + \tfrac{1}{3} \BNN_2 \BNN_{2,2} \smod{p^2}, \\
  \omega_5 &\equiv -\tfrac{1}{120} \BNN_1^5 - \tfrac{1}{6} \BNN_1^2 \BNN_{1,2} - \tfrac{1}{5} \BNN_{1,4} \smod{p}. 
\end{align*}
\end{theorem}
\medskip

\begin{theorem} \label{thm:main2}
Let $p \geq 11$ be a prime. Then we have
\begin{align*}
  \WQ_p &\equiv \sum_{\nu=1}^{6} \omega_\nu \, p^{\nu-1} \smod{p^6}
\shortintertext{and equivalently}
  (p-1)! &\equiv \sum_{\nu=0}^{6} \omega_\nu \, p^\nu \smod{p^7},
\end{align*}
where
\begin{align*}
  \omega_0 &= -1, \\
  \omega_1 &\equiv -6 \BNN_1 + 15 \BNN_2 - 20 \BNN_3 + 15 \BNN_4 - 6 \BNN_5 + \BNN_6 \smod{p^6}, \\
  \omega_2 &\equiv \BNN_1 (-\tfrac{13}{2} \BNN_1 + 15 \BNN_2 - 9 \BNN_3 + 2 \BNN_4)
  + \BNN_2 (-\tfrac{7}{2} \BNN_2 + 3 \BNN_4 - \BNN_5) - \tfrac{1}{2} \BNN_3^2 \smod{p^5}, \\
  \omega_3 &\equiv \BNN_1^2 (-\tfrac{10}{3} \BNN_1
  + \tfrac{15}{2} \BNN_2 - 3 \BNN_3 + \tfrac{1}{2} \BNN_4)
  + \BNN_2^2 (-3 \BNN_1 + \tfrac{1}{6} \BNN_2) + \BNN_1 \BNN_2 \BNN_3 \\
  &\quad - \tfrac{4}{3} \BNN_{1,2} + 2 \BNN_{2,2}
  - \tfrac{4}{3} \BNN_{3,2} + \tfrac{1}{3} \BNN_{4,2} \smod{p^4}, \\
  \omega_4 &\equiv \BNN_1^3 (-\tfrac{5}{8} \BNN_1 + \BNN_2 - \tfrac{1}{6} \BNN_3)
  - \tfrac{1}{4} \BNN_1^2 \BNN_2^2 
  - \BNN_1 \BNN_{1,2} + \BNN_2 \BNN_{2,2} - \tfrac{1}{3} \BNN_3 \BNN_{3,2} \smod{p^3}, \\
  \omega_5 &\equiv -\tfrac{1}{20} \BNN_1^5 + \tfrac{1}{24} \BNN_1^4 \BNN_2
  - \tfrac{1}{3}  \BNN_1 \BNN_2 \BNN_{1,2} - \tfrac{1}{2} \BNN_1^2 \BNN_{2,2}
  + \tfrac{2}{3} \BNN_1 \BNN_2 \BNN_{2,2}
  - \tfrac{2}{5} \BNN_{1,4} + \tfrac{1}{5} \BNN_{2,4} \smod{p^2}, \\
  \omega_6 &\equiv -\tfrac{1}{720} \BNN_1^6 - \tfrac{1}{18} \BNN_1^3 \BNN_{1,2}
  - \tfrac{1}{18} \BNN_{1,2}^2 - \tfrac{1}{5} \BNN_1 \BNN_{1,4} \smod{p}. 
\end{align*}
\end{theorem}
\medskip

The computation of the above congruences in Theorems~\ref{thm:main} and~\ref{thm:main2} 
depends on the Fermat quotients 
\begin{align*}
  q_p(a) &= \frac{a^{p-1} - 1}{p} \quad (1 \leq a < p) \\
\shortintertext{and their powers sums}
  Q_p(n) &= \sum_{a=1}^{p-1} q_p(a)^n \quad (n \geq 1).
\end{align*}
For congruences of the power sums $Q_p$ up to$\smod{p^4}$, see \cite{Kellner:2025b}.
Here, these congruences are extended up to$\smod{p^6}$, as follows.

\begin{theorem} \label{thm:main3}
Let $p$ be an odd prime. The congruences $\smod{p^5}$ and $\smod{p^6}$ hold for 
$p \geq 7$ and $p \geq 11$, respectively. We have
\begin{align*}
  Q_p(1) &\equiv (p-1) \BNN_1 - p^2 \, \BNN_{1,2} + \tfrac{11}{6} p^3 \, \BNN_{1,2}
  - p^4 (\BNN_{1,2} + \BNN_{1,4}) \\
  &\quad + p^5 (\tfrac{1}{6} \BNN_{1,2} + \tfrac{137}{60} \BNN_{1,4}) \smod{p^6} \\
  &\equiv (p-1) \BNN_1 - p^2 \, \BNN_{1,2} + \tfrac{11}{6} p^3 \, \BNN_{1,2}
  - p^4 (\BNN_{1,2} + \BNN_{1,4}) \smod{p^5}, \\
  \tfrac{1}{2} p \, Q_p(2) &\equiv (p-1) (\BNN_2 - \BNN_1) + p^2 (\BNN_{1,2} - 2 \BNN_{2,2})
  + p^3 (-\tfrac{11}{6} \BNN_{1,2} + \tfrac{13}{3} \BNN_{2,2}) \\
  &\quad + p^4 (\BNN_{1,2} - 3 \BNN_{2,2} + \BNN_{1,4} - 3 \BNN_{2,4})
  + p^5 (\tfrac{1}{2} \BNN_{1,2} + \tfrac{77}{12} \BNN_{1,4}) \smod{p^6} \\
  &\equiv (p-1) (\BNN_2 - \BNN_1) + p^2 (\BNN_{1,2} - 2 \BNN_{2,2}) \\
  &\quad + p^3 (-\tfrac{11}{6} \BNN_{1,2} + \tfrac{13}{3} \BNN_{2,2})
  - p^4 (2 \BNN_{1,2} + 2 \BNN_{1,4}) \smod{p^5}, \\
  \tfrac{1}{3} p^2 \, Q_p(3) &\equiv (p-1) (\BNN_3 - 2 \BNN_2 + \BNN_1) \\
  &\quad + p^2 (-\BNN_{1,2} + 4 \BNN_{2,2} - \tfrac{10}{3} \BNN_{3,2})
  + p^3 (\tfrac{11}{6} \BNN_{1,2} - \tfrac{26}{3} \BNN_{2,2} + \tfrac{47}{6} \BNN_{3,2}) \\
  &\quad + p^4 (5 \BNN_{1,2} - 6 \BNN_{2,2} + 6 \BNN_{1,4} - 8 \BNN_{2,4})
  + p^5 (\tfrac{1}{3} \BNN_{1,2} + \tfrac{47}{6} \BNN_{1,4}) \smod{p^6} \\
  &\equiv (p-1) (\BNN_3 - 2 \BNN_2 + \BNN_1)
  + p^2 (-\BNN_{1,2} + 4 \BNN_{2,2} - \tfrac{10}{3} \BNN_{3,2}) \\
  &\quad + p^3 (-6 \BNN_{1,2} + 7 \BNN_{2,2})
  - p^4 (\BNN_{1,2} + 2 \BNN_{1,4}) \smod{p^5}, \\
  \tfrac{1}{4} p^3 \, Q_p(4) &\equiv (p-1) (\BNN_4 - 3 \BNN_3 + 3 \BNN_2 - \BNN_1) \\
  &\quad + p^2 (\BNN_{1,2} - 6 \BNN_{2,2} + 10 \BNN_{3,2} - 5 \BNN_{4,2})
  + p^3 (\tfrac{21}{2} \BNN_{1,2} - 24 \BNN_{2,2} + \tfrac{27}{2} \BNN_{3,2}) \\
  &\quad + p^4 (3 \BNN_{1,2} - 3 \BNN_{2,2} + 8 \BNN_{1,4} - 9 \BNN_{2,4})
  + \tfrac{9}{2} p^5 \BNN_{1,4} \smod{p^6} \\
  &\equiv (p-1) (\BNN_4 - 3 \BNN_3 + 3 \BNN_2 - \BNN_1)
  + p^2 (-4 \BNN_{1,2} + 9 \BNN_{2,2} - 5 \BNN_{3,2}) \\
  &\quad + p^3 (-3 \BNN_{1,2} + 3 \BNN_{2,2})
  - p^4 \BNN_{1,4} \smod{p^5}, \\
  \tfrac{1}{5} p^4 \, Q_p(5) &\equiv (p-1) (\BNN_5 - 4 \BNN_4 + 6 \BNN_3 - 4 \BNN_2 + \BNN_1) \\
  &\quad + p^2 (6 \BNN_{1,2} - 20 \BNN_{2,2} + 22 \BNN_{3,2} - 8 \BNN_{4,2})
  + p^3 (6 \BNN_{1,2} - 12 \BNN_{2,2} + 6 \BNN_{3,2}) \\
  &\quad + p^4 (\tfrac{23}{5} \BNN_{1,4} - \tfrac{24}{5} \BNN_{2,4})
  + p^5 \BNN_{1,4} \smod{p^6} \\
  &\equiv -(\BNN_5 - 4 \BNN_4 + 6 \BNN_3 - 4 \BNN_2 + \BNN_1) \\
  &\quad + p^2 (-2 \BNN_{1,2} + 4 \BNN_{2,2} - 2 \BNN_{3,2})
  - \tfrac{1}{5} p^4 \BNN_{1,4} \smod{p^5}, \\
  \tfrac{1}{6} p^5 \, Q_p(6) &\equiv -(\BNN_6 - 5 \BNN_5 + 10 \BNN_4 - 10 \BNN_3 + 5 \BNN_2 - \BNN_1) \\
  &\quad + p^2 (\tfrac{10}{3} \BNN_{1,2} - 10 \BNN_{2,2} + 10 \BNN_{3,2} - \tfrac{10}{3} \BNN_{4,2})
  + p^4 (\BNN_{1,4} - \BNN_{2,4}) \smod{p^6}. 
\end{align*}
\end{theorem}
\medskip

The rest of the paper is organized as follows. The next section introduces 
properties of the Bernoulli numbers. Section~\ref{sec:congr} establishes a 
link between congruences of power sums, $Q_p$, and the Bernoulli numbers.
Section~\ref{trans} provides some preparations for the proofs of the main 
theorems that are subsequently given in Sections~\ref{sec:proof} and 
\ref{sec:proof2}, respectively. The last Section~\ref{sec:concl} ends with 
a conclusion.


\newpage
\section{Preliminaries}

Let $\NN$ be the set of positive integers. Let $\ZZ_p$ be the ring of $p$-adic 
integers. Define the linear $p$-adic backward shift operator by
\[
  \opsh_p: p \ZZ_p \to \ZZ_p, \quad
  a = \sum_{\nu \geq 1} a_\nu p^\nu \mapsto \frac{a}{p} = \sum_{\nu \geq 0} a_{\nu+1} p^\nu.
\]
This operator does not truncate a $p$-adic expansion, since it is defined on $p \ZZ_p$.
Define the linear forward difference operator $\opdf{h}{}$ with step $h$
and its powers by
\[
  \opdf{h}{n} f(s) = \sum_{\nu=0}^n \binom{n}{\nu} (-1)^{n-\nu} f(s+\nu h)
\]
for integers $n \geq 0$, $h \geq 1$, and any function $f: \ZZ_p \to \ZZ_p$. 
We use the expression, for example, $\opdf{h}{n} f(s+t) \valueat{s = 1}$
to indicate the variable and an initial value when needed.

The Bernoulli polynomials $\BN_n(x)$ are defined by the generating function
\begin{equation} \label{eq:bnp-gf}
  \frac{t e^{xt}}{e^t - 1} = \sum_{n \geq 0} \BN_n(x) \frac{t^n}{n!} \quad (|t| < 2\pi)
\end{equation}
and they are explicitly given by 
\begin{equation} \label{eq:bnp-def}
  \BN_n(x) = \sum_{k=0}^{n} \binom{n}{k} \BN_{n-k} \, x^k \quad (n \geq 0),
\end{equation}
where $\BN_n = \BN_n(0) \in \QQ$ is the $n$th Bernoulli number (cf.~\cite[Chap.\,9]{Cohen:2007}).
Define the $p$-integral Bernoulli numbers for odd prime $p$ and $n \in 2\NN_0$ by
\[
  \BNP_n =
  \begin{cases}
    0, & \text{if $n = 0$}; \\
    \BN_n + \tfrac{1}{p} - 1, & \text{if $n > 0$ and $p-1 \mid n$}; \\
    \BN_n, & \text{otherwise}.
  \end{cases}
\]
Further define the $p$-integral divided Bernoulli numbers by
\begin{equation} \label{eq:bnpd-def}
  \BNPD_n = \frac{\BNP_n}{n} \quad (n \geq 1) \andq
  \BNPD_n = 0 \quad (n \leq 0).
\end{equation}
Carlitz~\cite{Carlitz:1953} examined the properties of the numbers $\BNPD_n$
in case ${p-1 \mid n}$. As a result, the generalized Kummer congruences can be 
described in a combined way, as follows.

\begin{prop}[\cite{Kellner:2025b}] \label{prop:gen-congr}
Let $n \in 2\NN$, $p \geq 5$ be a prime, and $r \geq 1$. Then
\[
  \opdf{p-1}{r} \, \BNPD_\nu \valueat{\nu = n} \equiv 0 \smod{p^r}
\]
holds for the following two conditions:
\begin{enumerate}
\item $p-1 \nmid n$ and $n > r$;
\item $p-1 \mid n$ and $p > r + n/(p-1)$.
\end{enumerate}
\end{prop}


\section{Congruences of power sums}
\label{sec:congr}

Define the usual power sums by
\[
  S_n(m) = \sum_{\nu=1}^{m-1} \nu^n \quad (m, n \geq 0),
\]
where 
\begin{equation}\label{eq:sn-bn}
  S_n(x) = \frac{1}{n+1}(\BN_{n+1}(x) - \BN_{n+1}) \quad (n \geq 1). 
\end{equation}
Define the modified power sums by
\begin{equation} \label{eq:sh-def}
  \SH_n(x) = \frac{S_n(x) - S_0(x)}{x} \textq{where} \SH_0(p) = 0.
\end{equation}

\begin{prop}[\cite{Kellner:2025b}] 
For $n \geq 1$ and any prime $p$, we have 
\[
  Q_p(n) = \opsh_p^{n-1} \, \opdf{p-1}{n} \, \SH_\nu(p) \valueat{\nu = 0}.
\]
Let $r \in \set{5,6}$ and $p > r+1$ be a prime.
Let $1 \leq d \leq p$ and $n = d(p-1)$. Then we have
\[
  \SH_n(p) \equiv \BNP_n + p^2 \binom{n}{3} \BNPD_{n-2} + p^4 \binom{n}{5} \BNPD_{n-4} \smod{p^r}.
\]
\end{prop}

\begin{corl}[\cite{Kellner:2025b}] \label{cor:qp-diff}
Let $r \in \set{5,6}$, $1 \leq n \leq r$, and $p > r+1$ be a prime. Then we have
\begin{align}
  \frac{1}{n} p^{n-1} Q_p(n) &\equiv (p-1) \opdf{p-1}{n-1} \, \BNPD_\nu \valueat{\nu = p-1} \nonumber \\
  &\quad + \frac{1}{n} p^2 \, \opdf{p-1}{n} \binom{\nu}{3} \BNPD_{\nu-2} \valueat{\nu = 0} \label{eq:qp-diff} \\
  &\quad + \frac{1}{n} p^4 \, \opdf{p-1}{n} \binom{\nu}{5} \BNPD_{\nu-4} \valueat{\nu = 0} \smod{p^r},
  \label{eq:qp-diff-2}
\end{align}
where the case ${\SH_0(p) = 0}$ is handled by ${\BNPD_{j} = 0}$ for ${j < 0}$.
\end{corl}

\begin{lemma}[\cite{Kellner:2025b}] \label{lem:bin-diff}
For $k, n \geq 1$, and any prime $p > k$, we have 
\[
  \opdf{p-1}{n} \binom{\nu}{k} \valueat{\nu = 0}
  \equiv (-1)^k \binom{k-1}{n-1} \smod{p}.
\]
\end{lemma}

\begin{prop} \label{prop:qp-bnp}
Let $1 \leq n \leq 5$ and $p \geq 7$ be a prime. Then we have
\begin{align*}
  \frac{1}{n} p^{n-1} Q_p(n) &\equiv (p-1) \opdf{p-1}{n-1} \, \BNPD_\nu \valueat{\nu = p-1} \\
  &\quad + \frac{1}{n} p^2 \mleft( \alpha_n \, \BNPD_{p-3}
  + \alpha'_n \, \BNPD_{2(p-1)-2} + \alpha''_n \, \BNPD_{3(p-1)-2} \mright) \\
  &\quad + \frac{1}{n} p^3 \mleft( \beta_n \, \BNPD_{p-3} + \beta'_n \, \BNPD_{2(p-1)-2} \mright) \\
  &\quad + \frac{1}{n} p^4 \mleft( \gamma_n \, \BNPD_{p-3}
  + \delta_n \, \BNPD_{p-5} \mright) \smod{p^5},
\end{align*}
where the coefficients are given by 
\smallskip
\begin{center} \small
\setstretch{1.2}
\begin{tabular}{r@{$\,=\,$}l|r@{$\,=\,$}l|r@{$\,=\,$}l}
\toprule
  $\alpha$ & $(-1,2,-3,-16,-10)$, &
  $\beta$ & $(\frac{11}{6},-\frac{11}{3},-18,-12,0)$, &
  $\gamma$ & $-(1,4,3,0,0)$, \\
  $\alpha'$ & $(0,-4,12,36,20)$, &
  $\beta'$ & $(0,\frac{26}{3},21,12,0)$, &
  $\delta$ & $-(1,4,6,4,1)$. \\
  $\alpha''$ & $-(0,0,10,20,10)$, & \multicolumn{1}{c}{} & \\
\bottomrule
\end{tabular}
\end{center} 
\end{prop}

\begin{proof}
We use Corollary~\ref{cor:qp-diff} for ${r=5}$, where $n \leq r < p$.
We first evaluate \eqref{eq:qp-diff-2}, and obtain by Kummer congruences and 
$\BNPD_{-4} = \binom{0}{5} = 0$ that
\[
  \opdf{p-1}{n} \binom{\nu}{5} \BNPD_{\nu-4} \valueat{\nu = 0}
  \equiv \opdf{p-1}{n} \binom{\nu}{5} \BNPD_{p-5} \valueat{\nu = 0} \smod{p}
\]
and
\[
  \delta_n \equiv \opdf{p-1}{n} \binom{\nu}{5} \valueat{\nu = 0} \smod{p},
\]
where $\delta = -(1,4,6,4,1)$ by Lemma~\ref{lem:bin-diff}.
For \eqref{eq:qp-diff}, we substitute $\bb_j = \BNPD_{j(p-1)-2}$ and write
\[
  \opdf{p-1}{n} \binom{\nu}{3} \BNPD_{\nu-2} \valueat{\nu = 0}
  \equiv \sum_{j=1}^{n} (s_{n,j} + t_{n,j} \, p + u_{n,j} \, p^2) \, \bb_j \smod{p^3}
\]
with some integer coefficients $s_{n,j}$, $t_{n,j}$ and $u_{n,j}$. 
Using the Kummer congruences, we compute
\[
  \sum_{j=1}^{n} u_{n,j} \, \bb_j \equiv \gamma_n \, \bb_1 \smod{p}
\]
with $\gamma = -(1,4,3,0,0)$. 
By the generalized Kummer congruences, we have for all $j \geq 1$ that
\begin{align*}
  \bb_j - 2 \bb_{j+1} + \bb_{j+2} &\equiv 0 \smod{p^2}, \\
  \bb_j - 3 \bb_{j+1} + 3 \bb_{j+2} - \bb_{j+3} &\equiv 0 \smod{p^3}.
\end{align*}
For the coefficients $s_{n,j}$, we compute the following 
expressions and their reduction$\smod{p^3}$:
\begin{center} \small
\begin{tabular}{cl}
\toprule
  $n=1$ & $- \bb_1 \smod{p^3}$ \\
  $n=2$ & $2 \bb_1 - 4 \bb_2 \smod{p^3}$ \\
  $n=3$ & $-3 \bb_1 + 12 \bb_2 - 10 \bb_3 \smod{p^3}$ \\
  $n=4$ & $4 \bb_1 - 24 \bb_2 + 40 \bb_3 - 20 \bb_4
  \equiv -16 \bb_1 + 36 \bb_2 -20 \bb_3 \smod{p^3}$ \\
  $n=5$ & $-5 \bb_1 + 40 \bb_2 - 100 \bb_3 + 100 \bb_4 - 35 \bb_5
  \equiv -10 \bb_1 + 20 \bb_2 - 10 \bb_3 \smod{p^3}$. \\
\bottomrule
\end{tabular}
\end{center} 
This gives $\alpha = (-1,2,-3,-16,-10)$, $\alpha' = (0,-4,12,36,20)$, 
and $\alpha'' = -(0,0,10,20,10)$. Similarly, we obtain for the coefficients 
$t_{n,j}$ that
\begin{center} \small
\setstretch{1.2}
\begin{tabular}{cl}
\toprule
  $n=1$ & $\frac{11}{6} \bb_1 \smod{p^2}$ \\
  $n=2$ & $-\frac{11}{3} \bb_1 + \frac{26}{3} \bb_2 \smod{p^2}$ \\
  $n=3$ & $\frac{11}{2} \bb_1 - 26 \bb_2 + \frac{47}{2} \bb_3
  \equiv -18 \bb_1 + 21 \bb_2 \smod{p^2}$ \\
  $n=4$ & $-\frac{22}{3} \bb_1 + 52 \bb_2 - 94 \bb_3 + \frac{148}{3} \bb_4
  \equiv -12 \bb_1 + 12 \bb_2 \smod{p^2}$ \\
  $n=5$ & $\frac{55}{6} \bb_1 - \frac{260}{3} \bb_2 + 235 \bb_3 - \frac{740}{3} \bb_4 + \frac{535}{6} \bb_5
  \equiv 0 \smod{p^2}$. \\
\bottomrule
\end{tabular}
\end{center} 
This yields $\beta = (\frac{11}{6},-\frac{11}{3},-18,-12,0)$ and 
$\beta' = (0,\frac{26}{3},21,12,0)$, completing the proof.
\end{proof}

\begin{prop} \label{prop:qp-bnp-2}
Let $1 \leq n \leq 6$ and $p \geq 11$ be a prime. Then we have
\begin{align*}
  \frac{1}{n} p^{n-1} Q_p(n) &\equiv (p-1) \opdf{p-1}{n-1} \, \BNPD_\nu \valueat{\nu = p-1} \\
  &\quad + \frac{1}{n} p^2 \mleft( \alpha_n \, \BNPD_{p-3}
  + \alpha'_n \, \BNPD_{2(p-1)-2} + \alpha''_n \, \BNPD_{3(p-1)-2} + \alpha'''_n \, \BNPD_{4(p-1)-2} \mright) \\
  &\quad + \frac{1}{n} p^3 \mleft( \beta_n \, \BNPD_{p-3} + \beta'_n \, \BNPD_{2(p-1)-2} + \beta''_n \, \BNPD_{3(p-1)-2} \mright) \\
  &\quad + \frac{1}{n} p^4 \mleft( \gamma_n \, \BNPD_{p-3} + \gamma'_n \, \BNPD_{2(p-1)-2}
  + \epsilon_n \, \BNPD_{p-5} + \epsilon'_n \BNPD_{2(p-1)-4} \mright) \\
  &\quad + \frac{1}{n} p^5 \mleft( \delta_n \, \BNPD_{p-3}
  + \eta_n \, \BNPD_{p-5} \mright) \smod{p^6},
\end{align*}
where the coefficients are given by 
\begin{center} \small
\setstretch{1.2}
\begin{tabular}{r@{$\,=\,$}l|r@{$\,=\,$}l}
\toprule
  $\alpha$ & $(-1,2,-3,4,30,20)$, &
  $\beta$ & $(\frac{11}{6},-\frac{11}{3},\frac{11}{2},42,30,0)$, \\
  $\alpha'$ & $(0,-4,12,-24,-100,-60)$, &
  $\beta'$ & $(0,\frac{26}{3},-26,-96,-60,0)$, \\
  $\alpha''$ & $(0,0,-10,40,110,60)$, &
  $\beta''$ & $(0,0,\frac{47}{2},54,30,0)$, \\
  $\alpha'''$ & $-(0,0,0,20,40,20)$, & \multicolumn{1}{c}{} \\
\bottomrule
\end{tabular}
\end{center} 
and
\begin{center} \small
\setstretch{1.2}
\begin{tabular}{r@{$\,=\,$}l|r@{$\,=\,$}l}
\toprule
  $\gamma$ & $(-1,2,15,12,0,0)$, &
  $\epsilon$ & $(-1,2,18,32,23,6)$, \\
  $\gamma'$ & $-(0,6,18,12,0,0)$, &
  $\epsilon'$ & $-(0,6,24,36,24,6)$, \\
  $\delta$ & $(\frac{1}{6},1,1,0,0,0)$, &
  $\eta$ & $(\frac{137}{60},\frac{77}{6},\frac{47}{2},18,5,0)$. \\
\bottomrule
\end{tabular}
\end{center} 
\end{prop}

\begin{proof}
We use Corollary~\ref{cor:qp-diff} for ${r=6}$, where $n \leq r < p$.
We first evaluate \eqref{eq:qp-diff-2} and substitute $\bb_j = \BNPD_{j(p-1)-4}$. 
Note that $\BNPD_{-4} = \binom{0}{5} = 0$. We then obtain
\[
  \opdf{p-1}{n} \binom{\nu}{5} \BNPD_{\nu-4} \valueat{\nu = 0}
  \equiv \sum_{j=1}^{n} (d_{n,j} + e_{n,j} \, p) \, \bb_j \smod{p^2}
\]
with some integer coefficients $d_{n,j}$ and $e_{n,j}$. 
By the Kummer congruences, we determine that
\[
  \sum_{j=1}^{n} e_{n,j} \, \bb_j \equiv \eta_n \, \bb_1 \smod{p}
\]
with $\eta = (\frac{137}{60},\frac{77}{6},\frac{47}{2},18,5,0)$. 
By the generalized Kummer congruences, we have for all $j \geq 1$ 
the following cases (which we need later on) that
\begin{align*}
  \bb_j - 2 \bb_{j+1} + \bb_{j+2} &\equiv 0 \smod{p^2}, \\
  \bb_j - 3 \bb_{j+1} + 3 \bb_{j+2} - \bb_{j+3} &\equiv 0 \smod{p^3}, \\
  \bb_j - 4 \bb_{j+1} + 6 \bb_{j+2} - 4 \bb_{j+3} + \bb_{j+4} &\equiv 0 \smod{p^4}.
\end{align*}
For the coefficients $d_{n,j}$, we compute the following expressions and 
their reduction $\!\smod{p^2}$:
\begin{center} \small
\begin{tabular}{cl}
\toprule
  $n=1$ & $- \bb_1 \smod{p^2}$ \\
  $n=2$ & $2 \bb_1 - 6 \bb_2 \smod{p^2}$ \\
  $n=3$ & $-3 \bb_1 + 18 \bb_2 - 21 \bb_3 \equiv 18 \bb_1 - 24 \bb_2 \smod{p^2}$ \\
  $n=4$ & $4 \bb_1 - 36 \bb_2 + 84 \bb_3 - 56 \bb_4 \equiv 32 \bb_1 - 36 \bb_2 \smod{p^2}$ \\
  $n=5$ & $-5 \bb_1 + 60 \bb_2 - 210 \bb_3 + 280 \bb_4 - 126 \bb_5 \equiv 23 \bb_1 - 24 \bb_2 \smod{p^2}$ \\
  $n=6$ & $6 \bb_1 - 90 \bb_2 + 420 \bb_3 - 840 \bb_4 + 756 \bb_5 - 252 \bb_6 \equiv 6 \bb_1 - 6 \bb_2 \smod{p^2}$. \\
\bottomrule
\end{tabular}
\end{center} 
This yields $\epsilon=(-1,2,18,32,23,6)$ and $\epsilon'=-(0,6,24,36,24,6)$.

Next, we evaluate \eqref{eq:qp-diff} and substitute $\bb_j = \BNPD_{j(p-1)-2}$. 
With $\BNPD_{-2} = \binom{0}{3} = 0$, we obtain
\[
  \opdf{p-1}{n} \binom{\nu}{3} \BNPD_{\nu-2} \valueat{\nu = 0}
  \equiv \sum_{j=1}^{n} (s_{n,j} + t_{n,j} \, p + u_{n,j} \, p^2 + v_{n,j} \, p^3) \, \bb_j \smod{p^4}
\]
with some integer coefficients $s_{n,j}$, $t_{n,j}$, $u_{n,j}$ and $v_{n,j}$. 
Therefore, we have to spit the calculations into four cases.

Case $s_{n,j}$. We obtain the following expressions and their reduction$\smod{p^4}$:
\begin{center} \small
\begin{tabular}{cl}
\toprule
  $n=1$ & $- \bb_1 \smod{p^4}$ \\
  $n=2$ & $2 \bb_1 - 4 \bb_2 \smod{p^4}$ \\
  $n=3$ & $-3 \bb_1 + 12 \bb_2 - 10 \bb_3 \smod{p^4}$ \\
  $n=4$ & $4 \bb_1 - 24 \bb_2 + 40 \bb_3 - 20 \bb_4 \smod{p^4}$ \\
  $n=5$ & $-5 \bb_1 + 40 \bb_2 - 100 \bb_3 + 100 \bb_4 - 35 \bb_5
  \equiv 30 \bb_1 - 100 \bb_2 + 110 \bb_3 - 40 \bb_4 \smod{p^4}$ \\
  $n=6$ & $6 \bb_1 - 60 \bb_2 + 200 \bb_3 - 300 \bb_4 + 210 \bb_5 - 56 \bb_6
  \equiv 20 \bb_1 - 60 \bb_2 + 60 \bb_3 - 20 \bb_4 \smod{p^4}$. \\
\bottomrule
\end{tabular}
\end{center} 
We deduce that $\alpha = (-1,2,-3,4,30,20)$, $\alpha' = (0,-4,12,-24,-100,-60)$, \newline
$\alpha'' = (0,0,-10,40,110,60)$, and $\alpha''' = -(0,0,0,20,40,20)$.

For the coefficients $t_{n,j}$, it follows similarly that
\begin{center} \small
\setstretch{1.2}
\begin{tabular}{cl}
\toprule
  $n=1$ & $\frac{11}{6} \bb_1 \smod{p^3}$ \\
  $n=2$ & $-\frac{11}{3} \bb_1 + \frac{26}{3} \bb_2 \smod{p^3}$ \\
  $n=3$ & $\frac{11}{2} \bb_1 - 26 \bb_2 + \frac{47}{2} \bb_3 \smod{p^3}$ \\
  $n=4$ & $-\frac{22}{3} \bb_1 + 52 \bb_2 - 94 \bb_3 + \frac{148}{3} \bb_4
  \equiv 42 \bb_1 - 96 \bb_2 + 54 \bb_3 \smod{p^3}$ \\
  $n=5$ & $\frac{55}{6} \bb_1 - \frac{260}{3} \bb_2 + 235 \bb_3 - \frac{740}{3} \bb_4 + \frac{535}{6} \bb_5
  \equiv 30 \bb_1 - 60 \bb_2 + 30 \bb_3 \smod{p^3}$ \\
  $n=6$ & $-11 \bb_1 + 130 \bb_2 - 470 \bb_3 + 740 \bb_4 - 535 \bb_5 + 146 \bb_6 \equiv 0 \smod{p^3}$. \\
\bottomrule
\end{tabular}
\end{center} 
This defines $\beta = (\frac{11}{6},-\frac{11}{3},\frac{11}{2},42,30,0)$, 
$\beta' = (0,\frac{26}{3},-26,-96,-60,0)$, \newline
and $\beta'' = (0,0,\frac{47}{2},54,30,0)$.

For the coefficients $u_{n,j}$, we derive in a similar way that
\begin{center} \small
\setstretch{1.2}
\begin{tabular}{cl}
\toprule
  $n=1$ & $- \bb_1 \smod{p^2}$ \\
  $n=2$ & $2 \bb_1 - 6 \bb_2 \smod{p^2}$ \\
  $n=3$ & $-3 \bb_1 + 18 \bb_2 - 18 \bb_3
  \equiv 15 \bb_1 - 18 \bb_2 \smod{p^2}$ \\
  $n=4$ & $4 \bb_1 - 36 \bb_2 + 72 \bb_3 - 40 \bb_4
  \equiv 12 \bb_1 - 12 \bb_2 \smod{p^2}$ \\
  $n=5$ & $-5 \bb_1 + 60 \bb_2 - 180 \bb_3 + 200 \bb_4 - 75 \bb_5
  \equiv 0 \smod{p^2}$ \\
  $n=6$ & $6 \bb_1 - 90 \bb_2 + 360 \bb_3 - 600 \bb_4 + 450 \bb_5 - 126 \bb_6
  \equiv 0 \smod{p^2}$. \\
\bottomrule
\end{tabular}
\end{center} 
This implies that $\gamma = (-1,2,15,12,0,0)$ and $\gamma' = -(0,6,18,12,0,0)$.

Using the Kummer congruences, we finally compute
\[
  \sum_{j=1}^{n} v_{n,j} \, \bb_j \equiv \delta_n \, \bb_1 \smod{p}
\]
with $\delta = (\frac{1}{6},1,1,0,0,0)$. This completes the proof.
\end{proof}


\section{Transformations}
\label{trans}

The relationship between the Wilson quotient and the power sums $Q_p$ of the Fermat 
quotients is established by the following theorem.

\begin{theorem}[Kellner \cite{Kellner:2025}] \label{thm:kel}
We have the following statements:
\begin{enumerate}
\item There exist unique multivariate polynomials 
\[
  \psi_\nu(x_1,\ldots,x_\nu) \in \ZZ[x_1,\ldots,x_\nu] \quad (\nu \geq 1),
\]
which have no constant term and can be computed recursively; 

\item Let ${r \geq 1}$ and ${p > r}$ be an odd prime. Then we have 
\begin{align*}
  \WQ_p &\equiv \sum_{\nu=1}^{r} \frac{p^{\nu-1}}{\nu!} \, \psi_\nu(Q_p(1),\ldots,Q_p(\nu)) \smod{p^r} \\
\shortintertext{and equivalently}
  (p-1)! &\equiv -1 + \sum_{\nu=1}^{r} \frac{p^\nu}{\nu!} \, \psi_\nu(Q_p(1),\ldots,Q_p(\nu)) \smod{p^{r+1}}.
\end{align*}
\end{enumerate}
\end{theorem}

The first few polynomials $\psi_\nu$ are given below in Table~\ref{tab:psi}, 
which were computed in \cite{Kellner:2025}.

\begin{table}[H] \small
\setstretch{1.25}
\begin{center}
\begin{tabular}{r@{\;}c@{\;}l}
  \toprule
  $\psi_1$ & $=$ & $x_1$ \\
  $\psi_2$ & $=$ & $2 x_1 - x_1^2 - x_2$ \\
  $\psi_3$ & $=$ & $6 x_1 - 6 x_1^2 + x_1^3 + 3 x_1 x_2 - 3 x_2 + 2 x_3$ \\
  $\psi_4$ & $=$ & $24 x_1 - 36 x_1^2 + 12 x_1^3 - x_1^4 - 6 x_1^2 x_2 + 24 x_1 x_2 - 8 x_1 x_3 - 12 x_2 - 3 x_2^2 + 8 x_3 - 6 x_4$ \\
  $\psi_5$ & $=$ & $120 x_1 - 240 x_1^2 + 120 x_1^3 - 20 x_1^4 + x_1^5 + 10 x_1^3 x_2 - 90 x_1^2 x_2 + 20 x_1^2 x_3 + 180 x_1 x_2$ \\
  & & ${} + 15 x_1 x_2^2 - 80 x_1 x_3 + 30 x_1 x_4 - 60 x_2 - 30 x_2^2 + 20 x_2 x_3 + 40 x_3 - 30 x_4 + 24 x_5$ \\
  $\psi_6$ & $=$ & $720 x_1 - 1800 x_1^2 + 1200 x_1^3 - 300 x_1^4 + 30 x_1^5 - x_1^6
  - 15 x_1^4 x_2 + 240 x_1^3 x_2 - 40 x_1^3 x_3$ \\
  & & ${} - 1080 x_1^2 x_2 - 45 x_1^2 x_2^2 + 360 x_1^2 x_3 - 90 x_1^2 x_4 + 1440 x_1 x_2 + 270 x_1 x_2^2 - 120 x_1 x_2 x_3$ \\
  & & ${} - 720 x_1 x_3 + 360 x_1 x_4 - 144 x_1 x_5 - 360 x_2 - 270 x_2^2 - 15 x_2^3 + 240 x_2 x_3 - 90 x_2 x_4$ \\
  & & ${} + 240 x_3 - 40 x_3^2 - 180 x_4 + 144 x_5 - 120 x_6$ \\
  \bottomrule
\end{tabular}

\caption{First few polynomials $\psi_\nu$.}
\label{tab:psi}
\end{center}
\end{table}

In view of Corollary~\ref{cor:qp-diff}, we define for $n \geq 1$ and any odd prime 
$p$ the substitutions
\[
  Q_n = Q_p(n), \quad \QT_n = \frac{ p^{n-1}}{n} Q_p(n),
\]
and the polynomials
\[
  \PT_n(\QT_1,\ldots,\QT_n) =
  \frac{p^{n-1}}{n!} \, \psi_n(Q_1,\ldots,Q_n).
\]
We obtain a reformulation of Theorem~\ref{thm:kel}, as follows.

\begin{theorem} \label{thm:kel2}
Let ${r \geq 1}$ and ${p > r}$ be an odd prime. Then we have 
\begin{align*}
  \WQ_p &\equiv \sum_{\nu=1}^{r} \PT_\nu(\QT_1,\ldots,\QT_\nu) \smod{p^r} \\
\shortintertext{and equivalently}
  (p-1)! &\equiv -1 + p \sum_{\nu=1}^{r} \PT_\nu(\QT_1,\ldots,\QT_\nu) \smod{p^{r+1}}.
\end{align*}
\end{theorem}

The polynomials $\PT_\nu$, which are needed for later calculations, 
are listed in Table~\ref{tab:psi2} below.

\begin{table}[H] \small
\setstretch{1.2}
\begin{center}
\begin{tabular}{r@{\;}c@{\;}l}
  \toprule
  $\PT_1$ & $=$ & $x_1$ \\
  $\PT_2$ & $=$ & $p \, (x_1 - \tfrac{1}{2} x_1^2) - x_2$ \\
  $\PT_3$ & $=$ & $p^2 (x_1 -x_1^2 + \tfrac{1}{6} x_1^3)$ \\
  & & ${} + p \, (x_1 x_2 - x_2) + x_3$ \\
  $\PT_4$ & $=$ & $p^3 (x_1 - \tfrac{3}{2} x_1^2 + \tfrac{1}{2} x_1^3
  - \tfrac{1}{24} x_1^4)$ \\
  & & ${} + p^2 (2 x_1 x_2 - \tfrac{1}{2} x_1^2 x_2 - x_2)$ \\
  & & ${} + p \, (- \tfrac{1}{2} x_2^2 - x_1 x_3 + x_3) - x_4$ \\
  $\PT_5$ & $=$ & $p^4 (x_1 - 2 x_1^2 + x_1^3 - \tfrac{1}{6} x_1^4
  + \tfrac{1}{120} x_1^5)$ \\
  & & ${} + p^3 (3 x_1 x_2 - \tfrac{3}{2} x_1^2 x_2
  + \tfrac{1}{6} x_1^3 x_2 - x_2)$ \\
  & & ${} + p^2 (\tfrac{1}{2} x_1 x_2^2 - x_2^2 - 2 x_1 x_3
  + \tfrac{1}{2} x_1^2 x_3 + x_3)$ \\
  & & ${} + p \, (x_2 x_3 + x_1 x_4 - x_4) + x_5$ \\
  $\PT_6$ & $=$ & $p^5 (x_1 - \tfrac{5}{2} x_1^2 + \tfrac{5}{3} x_1^3 - \tfrac{5}{12} x_1^4
  + \tfrac{1}{24} x_1^5 - \tfrac{1}{720} x_1^6)$ \\
  & & ${} + p^4 (4 x_1 x_2 - 3 x_1^2 x_2 + \tfrac{2}{3} x_1^3 x_2 - \tfrac{1}{24} x_1^4 x_2 - x_2)$ \\
  & & ${} + p^3 (\tfrac{3}{2} x_1 x_2^2 - \tfrac{1}{4} x_1^2 x_2^2 - \tfrac{3}{2} x_2^2
  - 3 x_1 x_3 + \tfrac{3}{2} x_1^2 x_3 - \tfrac{1}{6} x_1^3 x_3 + x_3)$ \\
  & & ${} + p^2 (-x_1 x_2 x_3 + 2 x_2 x_3 - \tfrac{1}{6} x_2^3
  + 2 x_1 x_4 - \tfrac{1}{2} x_1^2 x_4 -x_4)$ \\
  & & ${} + p \, (-\tfrac{1}{2} x_3^2 - x_2 x_4 - x_1 x_5 + x_5) - x_6$ \\
  \bottomrule
\end{tabular}

\caption{First few polynomials $\PT_\nu$.}
\label{tab:psi2}
\end{center}
\end{table}

Recall the notation of \eqref{eq:bnn-1} and \eqref{eq:bnn-2}.
From \eqref{eq:bnpd-def}, it follows for $n \geq 1$ and any prime $p \geq 7$ that
\[
  \BNN_n = \BNPD_{n(p-1)} \andq \BNN_{n,d} = \BNPD_{n(p-1)-d} \textq{for} d \in \set{2,4}.
\]
These numbers lie in $\ZZ_p$ and satisfy the generalized Kummer congruences of 
Proposition~\ref{prop:gen-congr}. We rewrite Propositions~\ref{prop:qp-bnp} 
and \ref{prop:qp-bnp-2}, as follows.

\begin{lemma} \label{lem:qp-bnn}
Let $p \geq 7$ be a prime. Then we have
\begin{align*}
  \QT_1 &\equiv (p-1) \BNN_1
  - p^2 \, \BNN_{1,2}
  + \tfrac{11}{6} p^3 \, \BNN_{1,2}
  - p^4 (\BNN_{1,2} + \BNN_{1,4}) \smod{p^5}, \\
  \QT_2 &\equiv (p-1) (\BNN_2 - \BNN_1)
  + p^2 (\BNN_{1,2} - 2 \BNN_{2,2}) \\
  &\quad + p^3 (-\tfrac{11}{6} \BNN_{1,2} + \tfrac{13}{3} \BNN_{2,2})
  - p^4 (2 \BNN_{1,2} + 2 \BNN_{1,4}) \smod{p^5}, \\
  \QT_3 &\equiv (p-1) (\BNN_3 - 2 \BNN_2 + \BNN_1)
  + p^2 (-\BNN_{1,2} + 4 \BNN_{2,2} - \tfrac{10}{3} \BNN_{3,2}) \\
  &\quad + p^3 (-6 \BNN_{1,2} + 7 \BNN_{2,2})
  - p^4 (\BNN_{1,2} + 2 \BNN_{1,4}) \smod{p^5}, \displaybreak \\
  \QT_4 &\equiv (p-1) (\BNN_4 - 3 \BNN_3 + 3 \BNN_2 - \BNN_1)
  + p^2 (-4 \BNN_{1,2} + 9 \BNN_{2,2} - 5 \BNN_{3,2}) \\
  &\quad + p^3 (-3 \BNN_{1,2} + 3 \BNN_{2,2})
  - p^4 \BNN_{1,4} \smod{p^5}, \\
  \QT_5 &\equiv -(\BNN_5 - 4 \BNN_4 + 6 \BNN_3 - 4 \BNN_2 + \BNN_1) \\
  &\quad + p^2 (-2 \BNN_{1,2} + 4 \BNN_{2,2} - 2 \BNN_{3,2})
  - \tfrac{1}{5} p^4 \BNN_{1,4} \smod{p^5}. 
\end{align*}
\end{lemma}

\begin{proof}
This follows from Proposition~\ref{prop:qp-bnp}. The congruence for $\QT_5$ 
can be reduced (i.e., the term $p-1$ is replaced by $-1$)
by the generalized Kummer congruences of Proposition~\ref{prop:gen-congr}.
\end{proof}

\begin{lemma} \label{lem:qp-bnn-2}
Let $p \geq 11$ be a prime. Then we have
\begin{align*}
  \QT_1 &\equiv (p-1) \BNN_1
  - p^2 \, \BNN_{1,2}
  + \tfrac{11}{6} p^3 \, \BNN_{1,2} \\
  &\quad - p^4 (\BNN_{1,2} + \BNN_{1,4})
  + p^5 (\tfrac{1}{6} \BNN_{1,2} + \tfrac{137}{60} \BNN_{1,4}) \smod{p^6}, \\
  \QT_2 &\equiv (p-1) (\BNN_2 - \BNN_1)
  + p^2 (\BNN_{1,2} - 2 \BNN_{2,2})
  + p^3 (-\tfrac{11}{6} \BNN_{1,2} + \tfrac{13}{3} \BNN_{2,2}) \\
  &\quad + p^4 (\BNN_{1,2} - 3 \BNN_{2,2} + \BNN_{1,4} - 3 \BNN_{2,4})
  + p^5 (\tfrac{1}{2} \BNN_{1,2} + \tfrac{77}{12} \BNN_{1,4}) \smod{p^6}, \\
  \QT_3 &\equiv (p-1) (\BNN_3 - 2 \BNN_2 + \BNN_1)
  + p^2 (-\BNN_{1,2} + 4 \BNN_{2,2} - \tfrac{10}{3} \BNN_{3,2}) \\
  &\quad + p^3 (\tfrac{11}{6} \BNN_{1,2} - \tfrac{26}{3} \BNN_{2,2} + \tfrac{47}{6} \BNN_{3,2})
  + p^4 (5 \BNN_{1,2} - 6 \BNN_{2,2} + 6 \BNN_{1,4} - 8 \BNN_{2,4}) \\
  &\quad + p^5 (\tfrac{1}{3} \BNN_{1,2} + \tfrac{47}{6} \BNN_{1,4}) \smod{p^6}, \\
  \QT_4 &\equiv (p-1) (\BNN_4 - 3 \BNN_3 + 3 \BNN_2 - \BNN_1)
  + p^2 (\BNN_{1,2} - 6 \BNN_{2,2} + 10 \BNN_{3,2} - 5 \BNN_{4,2}) \\
  &\quad + p^3 (\tfrac{21}{2} \BNN_{1,2} - 24 \BNN_{2,2} + \tfrac{27}{2} \BNN_{3,2})
  + p^4 (3 \BNN_{1,2} - 3 \BNN_{2,2} + 8 \BNN_{1,4} - 9 \BNN_{2,4}) \\
  &\quad + \tfrac{9}{2} p^5 \BNN_{1,4} \smod{p^6}, \\
  \QT_5 &\equiv (p-1) (\BNN_5 - 4 \BNN_4 + 6 \BNN_3 - 4 \BNN_2 + \BNN_1)
  + p^2 (6 \BNN_{1,2} - 20 \BNN_{2,2} + 22 \BNN_{3,2} - 8 \BNN_{4,2}) \\
  &\quad + p^3 (6 \BNN_{1,2} - 12 \BNN_{2,2} + 6 \BNN_{3,2})
  + p^4 (\tfrac{23}{5} \BNN_{1,4} - \tfrac{24}{5} \BNN_{2,4})
  + p^5 \BNN_{1,4} \smod{p^6}, \\
  \QT_6 &\equiv -(\BNN_6 - 5 \BNN_5 + 10 \BNN_4 - 10 \BNN_3 + 5 \BNN_2 - \BNN_1) \\
  &\quad + p^2 (\tfrac{10}{3} \BNN_{1,2} - 10 \BNN_{2,2} + 10 \BNN_{3,2} - \tfrac{10}{3} \BNN_{4,2})
  + p^4 (\BNN_{1,4} - \BNN_{2,4}) \smod{p^6}. 
\end{align*}
\end{lemma}

\begin{proof}
This follows from Proposition~\ref{prop:qp-bnp-2}. The congruence for $\QT_6$ 
is reduced by the generalized Kummer congruences of Proposition~\ref{prop:gen-congr}.
\end{proof}

\begin{proof}[Proof of Theorem~\ref{thm:main3}]
This follows from Lemmas~\ref{lem:qp-bnn} and \ref{lem:qp-bnn-2}.
\end{proof}


\section{Proof of Theorem~\ref{thm:main}}
\label{sec:proof}

We can apply Lemma~\ref{lem:qp-bnn} and \cite[Theorem~1.5]{Kellner:2025b} to 
substitute values of $\QT_\nu$ in different moduli. We use \textsl{Mathematica} 
to shorten lengthy calculations. The remaining terms are then simplified and 
reduced by applying the generalized Kummer congruences afterwards. 
This procedure shall eliminate several redundant terms. In particular,
we can manipulate terms by adding zero sums, e.g.,
\[
  \sum_{\nu} \lambda_\nu \, \mathcal{E}_\nu \equiv 0 \smod{p^r},
\]
where $r \geq 1$, $\lambda_\nu \in \ZZ_p$, and $\mathcal{E}_\nu \equiv 0 \smod{p^r}$ 
are expressions in terms of the divided Bernoulli numbers.

\begin{proof}[Proof of Theorem~\ref{thm:main}]
By Theorem~\ref{thm:kel2}, we have
\[
  \WQ_p \equiv \sum_{\nu=1}^{5} \PT_\nu(\QT_1,\ldots,\QT_\nu) \smod{p^5}
\]
with polynomials $\PT_\nu$ of Table~\ref{tab:psi2}. After substituting 
values of $\QT_\nu$ and rearranging of terms, we rewrite the sum as
\[
  \WQ_p \equiv \sum_{\nu=1}^{5} \omega_\nu \, p^{\nu-1} \smod{p^5}.
\]
The computation leads to the following terms that are further simplified.
We first obtain
\[
  \omega_1 \equiv -5 \BNN_1 + 10 \BNN_2 - 10 \BNN_3 + 5 \BNN_4 - \BNN_5 \smod{p^5}
\]
and
\begin{align*}
  \omega_2 &\equiv \BNN_1 (-5 \BNN_1 + 10 \BNN_2 - 5 \BNN_3 + \BNN_4)
  + \BNN_2 (-\tfrac{5}{2} \BNN_2 + \BNN_3) \\
  &\equiv -\tfrac{5}{2} \BNN_1^2 + \tfrac{15}{2} \BNN_2^2
  + \tfrac{5}{2} \BNN_3^2 + \BNN_1 \BNN_4 - 9 \BNN_2 \BNN_3 \smod{p^4},
\end{align*}
where the latter congruence follows from adding the expression
\[
  \tfrac{5}{2}(\BNN_1 - 2 \BNN_2 + \BNN_3)^2 \equiv 0 \smod{p^4}.
\]

Further, we have
\[
  \omega_3 \equiv \omega_{3,1} - \BNN_{1,2} + \BNN_{2,2} - \tfrac{1}{3} \BNN_{3,2} \smod{p^3}, 
\]
where
\begin{align*}
  \omega_{3,1} &\equiv \BNN_1 (3 \BNN_1 - 9 \BNN_2 + 5 \BNN_3 - \BNN_4)
  + \BNN_1^2 (-\tfrac{5}{3} \BNN_1 + \tfrac{5}{2} \BNN_2 - \tfrac{1}{2} \BNN_3) \\
  &\quad + \BNN_2 (-\tfrac{1}{2} \BNN_1 \BNN_2 + 4 \BNN_2 - 2 \BNN_3) \\
  &\equiv \BNN_1 (2 \BNN_1 - 6 \BNN_2 + 2 \BNN_3) + \cdots \smod{p^3}
\end{align*}
with the remaining terms left unchanged. Applying the expression
\[
  -2(\BNN_1 - \BNN_2)(\BNN_1 - 2 \BNN_2 + \BNN_3) \equiv 0 \smod{p^3}
\] 
yields
\[
  \omega_{3,1} \equiv -\tfrac{1}{2} \BNN_1 \BNN_2^2
  - \BNN_1^2 (\tfrac{5}{3} \BNN_1 - \tfrac{5}{2} \BNN_2 + \tfrac{1}{2} \BNN_3) \smod{p^3}. 
\]

Next, we have
\[
  \omega_4 \equiv -\tfrac{5}{24} \BNN_1^4 + \tfrac{1}{6} \BNN_1^3 \BNN_2
  + \omega_{4,1} + \omega_{4,2} \smod{p^2}, 
\]
where
\begin{align*}
  \omega_{4,1} &\equiv \tfrac{3}{2} \BNN_1^2 - 2 \BNN_1 \BNN_2 - \tfrac{1}{2} \BNN_2^2 + \BNN_2 \BNN_3
  + \BNN_1 (\tfrac{5}{2} \BNN_1^2 - 5 \BNN_1 \BNN_2 + \BNN_1 \BNN_3 + \tfrac{3}{2} \BNN_2^2) \\
  &\equiv (\tfrac{3}{2} \BNN_1^2 - 3 \BNN_1 \BNN_2 + \tfrac{3}{2} \BNN_2^2) (1 + \BNN_1) \\
  &\equiv \tfrac{3}{2} (\BNN_1 - \BNN_2)^2 (1 + \BNN_1) \\
  &\equiv 0 \smod{p^2} 
\end{align*} 
and
\begin{align*}
  \omega_{4,2} &\equiv \BNN_{1,2} (-\tfrac{10}{3} \BNN_1 + \tfrac{20}{3} \BNN_2 - 5 \BNN_3 + \BNN_4)
  + \BNN_{2,2} (\tfrac{5}{3} \BNN_1 - \tfrac{10}{3} \BNN_2 + 2 \BNN_3) \\
  &\equiv -\tfrac{1}{3} \BNN_{1,2} (\BNN_1 + \BNN_2) - \tfrac{1}{3} \BNN_{2,2} (\BNN_1 - 2 \BNN_2) \smod{p^2}.
\end{align*}
From adding 
\[
  -\tfrac{1}{3}(\BNN_{1,2}-\BNN_{2,2})(\BNN_1 - \BNN_2) \equiv 0 \smod{p^2},
\]
we then derive that
\[
  \omega_{4,2} \equiv -\tfrac{2}{3} \BNN_1 \BNN_{1,2} + \tfrac{1}{3} \BNN_2 \BNN_{2,2} \smod{p^2}.
\]

Finally, we get $21$ terms for $\omega_5$, where most of them are directly canceled 
out by the Kummer congruences. Therefore, we only write the remaining terms down as
\begin{equation} \label{eq:w5-p1}
  \omega_5 \equiv -\tfrac{1}{120} \BNN_1^5 - \tfrac{1}{6} \BNN_1^2 \BNN_{1,2} - \tfrac{1}{5} \BNN_{1,4} \smod{p}.
\end{equation}
This gives the result and completes the proof.
\end{proof}


\section{Proof of Theorem~\ref{thm:main2}}
\label{sec:proof2}

We use Lemma~\ref{lem:qp-bnn-2} together with Lemma~\ref{lem:qp-bnn} and 
\cite[Theorem~1.5]{Kellner:2025b} to substitute values of $\QT_\nu$ in different 
moduli. We proceed similarly as in the proof of Theorem~\ref{thm:main}.

\begin{proof}[Proof of Theorem~\ref{thm:main2}]
By Theorem~\ref{thm:kel2}, we have
\[
  \WQ_p \equiv \sum_{\nu=1}^{6} \PT_\nu(\QT_1,\ldots,\QT_\nu) \smod{p^6}
\]
with polynomials $\PT_\nu$ of Table~\ref{tab:psi2}. After substituting 
values of $\QT_\nu$ and rearranging of terms, we rewrite the sum as
\[
  \WQ_p \equiv \sum_{\nu=1}^{6} \omega_\nu \, p^{\nu-1} \smod{p^6}.
\]

We then derive
\[
  \omega_1 \equiv -6 \BNN_1 + 15 \BNN_2 - 20 \BNN_3 + 15 \BNN_4 - 6 \BNN_5 + \BNN_6 \smod{p^6}
\]
and
\begin{align*}
  \omega_2 &\equiv \BNN_1 (-\tfrac{15}{2} \BNN_1 + 20 \BNN_2 - 15 \BNN_3 + 6 \BNN_4 - \BNN_5)
  + \BNN_2 (-\tfrac{15}{2} \BNN_2 + 6 \BNN_3 - \BNN_4) - \tfrac{1}{2} \BNN_3^2 \\
  &\equiv \BNN_1 (-\tfrac{13}{2} \BNN_1 + 15 \BNN_2 - 9 \BNN_3 + 2 \BNN_4)
  + \BNN_2 (-\tfrac{7}{2} \BNN_2 + 3 \BNN_4 - \BNN_5) - \tfrac{1}{2} \BNN_3^2
  \smod{p^5}
\end{align*}
from applying the expression 
\[
  (\BNN_1 -\BNN_2)(\BNN_1 - 4 \BNN_2 + 6 \BNN_3 - 4 \BNN_4 + \BNN_5) \equiv 0 \smod{p^5}.
\]

Next, we have
\[
  \omega_3 \equiv \omega_{3,1} + \omega_{3,2} + \omega_{3,3} \smod{p^4}
\]
with
\begin{align*}
  \omega_{3,1} &\equiv \BNN_1 (4 \BNN_1 - 16 \BNN_2 + 14 \BNN_3 - 6 \BNN_4 + \BNN_5) \\
  &\quad + \BNN_2 (10 \BNN_2 - 10 \BNN_3 + 2 \BNN_4) + \BNN_3^2 \\
  &\equiv \BNN_1 (3 \BNN_1 - 12 \BNN_2 + 8 \BNN_3 - 2 \BNN_4) + \cdots \smod{p^4}, 
\end{align*}
where the remaining terms of $\omega_{3,1}$ are unchanged. By adding the expressions
\[
  -(\BNN_1 - 2 \BNN_2 + \BNN_3)^2 - 2 (\BNN_1 - \BNN_2)
  (\BNN_1 - 3 \BNN_2 + 3 \BNN_3 - \BNN_4) \equiv 0 \smod{p^4},
\]
we obtain
\[
  \omega_{3,1} \equiv 0 \smod{p^4}.
\]
The remaining parts are
\[
  \omega_{3,2} \equiv \BNN_1^2 (-\tfrac{10}{3} \BNN_1
  + \tfrac{15}{2} \BNN_2 - 3 \BNN_3 + \tfrac{1}{2} \BNN_4)
  + \BNN_2^2 (-3 \BNN_1 + \tfrac{1}{6} \BNN_2) + \BNN_1 \BNN_2 \BNN_3 \smod{p^4} 
\]
and
\[
  \omega_{3,3} \equiv -\tfrac{4}{3} \BNN_{1,2} + 2 \BNN_{2,2}
  - \tfrac{4}{3} \BNN_{3,2} + \tfrac{1}{3} \BNN_{4,2} \smod{p^4}.
\]

In case of $\omega_4$, we collect the following terms such that
\[
  \omega_4 \equiv \omega_{4,1} + \omega_{4,2} + \omega_{4,3} \smod{p^3}.
\]
For the first part, we have
\begin{align*}
  \omega_{4,1} &\equiv \BNN_1 (\tfrac{5}{2} \BNN_1 - 6 \BNN_2 + 2 \BNN_3)
  + \BNN_1^2 (\tfrac{9}{2} \BNN_1 - \tfrac{27}{2} \BNN_2 + 6 \BNN_3 - \BNN_4) \\
  &\quad + \BNN_2 (\BNN_2 + 2 \BNN_3 - \BNN_4 - 3 \BNN_1 \BNN_3)
  + \BNN_2^2 (\tfrac{15}{2} \BNN_1 - \tfrac{1}{2} \BNN_2) - \tfrac{1}{2} \BNN_3^2 \\
  &\equiv \BNN_1 (\tfrac{5}{2} \BNN_1 - 7 \BNN_2 + 2 \BNN_3)
  + \BNN_1^2 (\tfrac{7}{2} \BNN_1 - \tfrac{21}{2} \BNN_2 + 3 \BNN_3) \\
  &\quad + \BNN_2 (4 \BNN_2 - \BNN_3 - 3 \BNN_1 \BNN_3)
  + \BNN_2^2 (\tfrac{15}{2} \BNN_1 - \tfrac{1}{2} \BNN_2) - \tfrac{1}{2} \BNN_3^2 \smod{p^3}. 
\end{align*}
Adding the expressions 
\[
  \tfrac{1}{2} (\BNN_1 - 2 \BNN_2 + \BNN_3)^2 - \tfrac{1}{2} (\BNN_1 - \BNN_2)^3
  - 3 (\BNN_1 - \BNN_2) (\BNN_1 - 2 \BNN_2 + \BNN_3) (1 + \BNN_1) \equiv 0 \smod{p^3}
\]
yields
\[
  \omega_{4,1} \equiv 0 \smod{p^3}.
\]
The second part gives
\begin{align*}
  \omega_{4,2} &\equiv \BNN_1^3 (-\tfrac{5}{8} \BNN_1 + \BNN_2 - \tfrac{1}{6} \BNN_3)
  - \tfrac{1}{4} \BNN_1^2 \BNN_2^2 \smod{p^3}.
\end{align*}
The last part is
\begin{align*}
  \omega_{4,3} &\equiv \BNN_{1,2} (-6 \BNN_1 + 15 \BNN_2 - 15 \BNN_3 + 6 \BNN_4 - \BNN_5) \\
  &\quad + \BNN_{2,2} (6 \BNN_1 - 15 \BNN_2 + 12 \BNN_3 - 2 \BNN_4) \\
  &\quad + \BNN_{3,2} (-2 \BNN_1 + 5 \BNN_2 - \tfrac{10}{3} \BNN_3) \\
  &\equiv \BNN_{1,2} (-3 \BNN_1 + 5 \BNN_2 - 3 \BNN_3) \\
  &\quad + \BNN_{2,2} (4 \BNN_1 - 9 \BNN_2 + 6 \BNN_3) \\
  &\quad + \BNN_{3,2} (-2 \BNN_1 + 5 \BNN_2 - \tfrac{10}{3} \BNN_3) \smod{p^3}.
\end{align*}
Adding the expression 
\[
  \mleft(3 (\BNN_1 - 2 \BNN_2 + \BNN_3) - (\BNN_1 - \BNN_2) \mright)
  (\BNN_{1,2} - 2 \BNN_{2,2} + \BNN_{3,2}) \equiv 0 \smod{p^3}
\]
finally provides
\[
  \omega_{4,3} \equiv -\BNN_1 \BNN_{1,2} + \BNN_2 \BNN_{2,2} - \tfrac{1}{3} \BNN_3 \BNN_{3,2}
\]

For $\omega_5$, we obtain several terms that eventually vanish. Therefore, 
we group the corresponding terms in the following way such that
\begin{align*}
  \omega_5 &\equiv -\tfrac{1}{20} \BNN_1^5 + \tfrac{1}{24} \BNN_1^4 \BNN_2
  + \omega_{5,1} + \omega_{5,2} + \omega_{5,3}
  - \tfrac{2}{5} \BNN_{1,4} + \tfrac{1}{5} \BNN_{2,4} \smod{p^2}
\end{align*}
with $\omega_{5,1}$, $\omega_{5,2}$, and $\omega_{5,3}$ as given below.
It is easy to verify that
\begin{align*}
  \omega_{5,1} &\equiv \BNN_1 (\BNN_1 - \tfrac{9}{2} \BNN_2^2 + \BNN_1 \BNN_2^2)
  + \BNN_2 (- 2 \BNN_2 + \tfrac{1}{2} \BNN_2^2) \\
  &\quad + \BNN_1^2 (\tfrac{1}{2} \BNN_1 + 3 \BNN_2 - 3 \BNN_3 + \tfrac{1}{2} \BNN_4)
  + \BNN_1^3 (\tfrac{3}{2} \BNN_1 - 3 \BNN_2 + \tfrac{1}{2} \BNN_3) \\
  &\quad + \BNN_3 (-\BNN_1 + 2 \BNN_2 + 3 \BNN_1 \BNN_2) \\
  &\equiv \BNN_1 (\BNN_1 - \tfrac{9}{2} \BNN_2^2 + \BNN_1 \BNN_2^2)
  + \BNN_2 (- 2 \BNN_2 + \tfrac{1}{2} \BNN_2^2) \\
  &\quad + \BNN_1^2 (\tfrac{5}{2} \BNN_1 - \tfrac{3}{2} \BNN_2)
  + \BNN_1^3 (\BNN_1 - 2 \BNN_2) \\
  &\quad + (-\BNN_1 + 2 \BNN_2)(-\BNN_1 + 2 \BNN_2 + 3 \BNN_1 \BNN_2) \\
  &\equiv \tfrac{1}{2} (\BNN_1 - \BNN_2)^2 (4 + 5 \BNN_1 + 2 \BNN_1^2 + \BNN_2) \\
  &\equiv 0 \smod{p^2}.
\end{align*}
Further, we have
\begin{align*}
  \omega_{5,2} &\equiv \BNN_{1,2} (\tfrac{17}{6} \BNN_1 - 12 \BNN_2 + \tfrac{56}{3} \BNN_3 - 11 \BNN_4 + \tfrac{11}{6} \BNN_5) \\
  &\quad + \BNN_{2,2} (-5 \BNN_1 + \tfrac{49}{3} \BNN_2 - \tfrac{55}{3} \BNN_3 + \tfrac{19}{3} \BNN_4) \\
  &\quad + \BNN_{3,2} (2 \BNN_1 - 5 \BNN_2 + \tfrac{10}{3} \BNN_3) \\
  &\equiv \BNN_{1,2} (\tfrac{2}{3} \BNN_1 - \tfrac{1}{3} \BNN_2)
  + \BNN_{2,2} (\tfrac{2}{3} \BNN_1 - \tfrac{4}{3} \BNN_2) \\
  &\quad + (-\BNN_{1,2} + 2 \BNN_{2,2}) (-\tfrac{4}{3} \BNN_1 + \tfrac{5}{3} \BNN_2) \\
  &\equiv 2 (\BNN_{1,2} - \BNN_{2,2}) (\BNN_1 - \BNN_2) \\
  &\equiv 0 \smod{p^2}.
\end{align*}
It remains that
\begin{align*}
  \omega_{5,3} &\equiv \BNN_{1,2} (-5 \BNN_1^2 + \tfrac{35}{3} \BNN_1 \BNN_2
  - 6 \BNN_1 \BNN_3 + \BNN_1 \BNN_4 - 3 \BNN_2^2 + \BNN_2 \BNN_3) \\
  &\quad + \BNN_{2,2} (\tfrac{5}{2} \BNN_1^2 - \tfrac{16}{3} \BNN_1 \BNN_2 + 2 \BNN_1 \BNN_3 + \BNN_2^2) \\
  &\equiv \BNN_{1,2} (-\BNN_1^2 + \tfrac{5}{3} \BNN_1 \BNN_2 - \BNN_2^2)
  + \BNN_{2,2} (\tfrac{1}{2} \BNN_1^2 - \tfrac{4}{3} \BNN_1 \BNN_2 + \BNN_2^2) \\
  &\equiv -\tfrac{1}{3} \BNN_1 \BNN_2 \BNN_{1,2}
  - \tfrac{1}{2} \BNN_1^2 \BNN_{2,2} + \tfrac{2}{3} \BNN_1 \BNN_2 \BNN_{2,2} \smod{p^2},
\end{align*}
where the last congruence follows from adding the expression
\[
  (\BNN_{1,2} - \BNN_{2,2}) (\BNN_1 - \BNN_2)^2 \equiv 0 \smod{p^2}.
\]
As a result, this finally yields
\begin{equation} \label{eq:w5-p2}
\begin{aligned}
  \omega_5 &\equiv -\tfrac{1}{20} \BNN_1^5 + \tfrac{1}{24} \BNN_1^4 \BNN_2
  - \tfrac{1}{3} \BNN_1 \BNN_2 \BNN_{1,2} - \tfrac{1}{2} \BNN_1^2 \BNN_{2,2} \\
  &\quad + \tfrac{2}{3} \BNN_1 \BNN_2 \BNN_{2,2}
  - \tfrac{2}{5} \BNN_{1,4} + \tfrac{1}{5} \BNN_{2,4} \smod{p^2}.
\end{aligned}
\end{equation}

At the end, we collect $48$ terms for $\omega_6$, where almost all terms coincide 
or vanish by the Kummer congruences. Only $4$ terms remain such that
\begin{align*}
  \omega_6 &\equiv -\tfrac{1}{720} \BNN_1^6 - \tfrac{1}{18} \BNN_1^3 \BNN_{1,2}
  - \tfrac{1}{18} \BNN_{1,2}^2 - \tfrac{1}{5} \BNN_1 \BNN_{1,4} \smod{p}.
\end{align*}
This implies the result, completing the proof.
\end{proof}


\section{Conclusion}
\label{sec:concl}

Let $r \geq 1$ and $p > r$ be an odd prime. Then we have
\begin{align*}
  \WQ_p &\equiv \sum_{\nu=1}^{r} \omega_\nu \, p^{\nu-1} \smod{p^r} \\
\shortintertext{and equivalently}
  (p-1)! &\equiv \sum_{\nu=0}^{r} \omega_\nu \, p^\nu \smod{p^{r+1}}
\end{align*}
with $\omega_0 = -1$ and some coefficients $\omega_\nu \in \ZZ_p$.

It turns out that the computation of the coefficients $\omega_\nu$ for increasing 
$r$ becomes more and more difficult compared with the relatively simple cases for 
$r \in \set{1,2,3,4}$, as computed in~\cite{Kellner:2025b}. On the one side, 
the number of terms of the polynomials $\psi_\nu$ of Theorem~\ref{thm:kel}
increases exponentially for increasing $\nu$, as discussed in \cite{Kellner:2025}.
On the other side, the congruences of the power sums $Q_p$ of the Fermat quotients 
have to be computed in higher moduli for increasing~$r$ and so the number of their 
terms also increases.

For example, we consider the computation of $\omega_5 \smod{p^2}$, 
see \eqref{eq:w5-p2} and above in the proof of Theorem~\ref{thm:main2}.
Several terms of the computation of $\omega_5 \smod{p^2}$ eventually vanish 
by applying the generalized Kummer congruences. As a hint which terms should vanish, 
one can look at the reduction $\omega_5 \smod{p}$, as computed in \eqref{eq:w5-p1}.
Table~\ref{tab:red} below describes the related terms of this reduction from 
modulo $p^2$ to $p$ for $p \geq 7$.

\begin{table}[H] \small
\setstretch{1.5}
\begin{center}
\begin{tabular}{c@{\;\;}c@{\;}c}
  \toprule
  $\omega_5 \smod{p^2}$ && $\omega_5 \smod{p}$ \\
  \midrule
  $-\tfrac{1}{20} \BNN_1^5 + \tfrac{1}{24} \BNN_1^4 \BNN_2$ & $\mapsto$ & $-\tfrac{1}{120} \BNN_1^5$ \\
  $-\tfrac{1}{3} \BNN_1 \BNN_2 \BNN_{1,2} - \tfrac{1}{2} \BNN_1^2 \BNN_{2,2}
  + \tfrac{2}{3} \BNN_1 \BNN_2 \BNN_{2,2}$ & $\mapsto$ & $-\tfrac{1}{6} \BNN_1^2 \BNN_{1,2}$ \\
  $-\tfrac{2}{5} \BNN_{1,4} + \tfrac{1}{5} \BNN_{2,4}$ & $\mapsto$ & $-\tfrac{1}{5} \BNN_{1,4}$ \\
  \bottomrule
\end{tabular}

\caption{Reduction of terms of $\omega_5$ for $p \geq 7$.}
\label{tab:red}
\end{center}
\end{table}

For any suitable exponents $r, s \geq 2$ and index $\nu \geq 1$, we have a chain 
of reductions, e.g.,
\begin{alignat*}{2}
  \WQ_p \smod{p^{r+1}} \quad \longmapsto \quad & \WQ_p \smod{p^r} & \quad \longmapsto \quad & \WQ_p \smod{p^{r-1}} \\
\shortintertext{and}
  \omega_\nu \smod{p^{s+1}} \quad \longmapsto \quad & \omega_\nu \smod{p^s} & \quad \longmapsto \quad & \omega_\nu \smod{p^{s-1}},
\end{alignat*}
respectively. The coefficients $\omega_\nu$ coincide by reduction in any moduli 
and can thus be verified modulo~$p$ by entries of Table~\ref{tab:omega} below 
(values of $\omega_1, \ldots, \omega_4 \smod{p}$ are taken from \cite{Kellner:2025b}).
Moreover, the terms of these entries reveal a certain pattern such that 
${\omega_0 = -1}$ fits perfectly. Roughly speaking, one may conjecture that 
the elements of the sequences, e.g.,  
\[
  \left( -\tfrac{1}{\ell!} \BNN_1^\ell \right)_{\ell \geq 0}\!, \quad
  \left( -\tfrac{1}{3 \cdot \ell!} \BNN_1^\ell \BNN_{1,2} \right)_{\ell \geq 0}\!, \andq
  \left( -\tfrac{1}{5 \cdot \ell!} \BNN_1^\ell \BNN_{1,4} \right)_{\ell \geq 0}\!
\]
occur as terms of successive coefficients $\omega_\nu \smod{p}$, see Table~\ref{tab:omega}.

\begin{table}[H] \small
\setstretch{1.5}
\begin{center}
\begin{tabular}{r@{\;}c@{\;}l}
  \toprule
  $\omega_0$ & $\equiv$ & $-1 \smod{p}$ \\
  $\omega_1$ & $\equiv$ & $-\BNN_1 \smod{p}$ \\
  $\omega_2$ & $\equiv$ & $-\tfrac{1}{2} \BNN_1^2 \smod{p}$ \\
  $\omega_3$ & $\equiv$ & $-\tfrac{1}{6} \BNN_1^3 - \tfrac{1}{3} \BNN_{1,2} \smod{p}$ \\
  $\omega_4$ & $\equiv$ & $-\tfrac{1}{24} \BNN_1^4 - \tfrac{1}{3} \BNN_1 \BNN_{1,2} \smod{p}$ \\
  $\omega_5$ & $\equiv$ & $-\tfrac{1}{120} \BNN_1^5 - \tfrac{1}{6} \BNN_1^2 \BNN_{1,2} - \tfrac{1}{5} \BNN_{1,4} \smod{p}$ \\
  $\omega_6$ & $\equiv$ & $-\tfrac{1}{720} \BNN_1^6 - \tfrac{1}{18} \BNN_1^3 \BNN_{1,2}
  - \tfrac{1}{5} \BNN_1 \BNN_{1,4} - \tfrac{1}{18} \BNN_{1,2}^2 \smod{p}$ \\
  \bottomrule
\end{tabular}

\caption{Coefficients $\omega_\nu \smod{p}$ for $p \geq 7$.}
\label{tab:omega}
\end{center}
\end{table}


\bibliographystyle{amsplain}

\end{document}